\documentclass[journal,twoside,print]{ieeecolor}
\usepackage{generic}
\usepackage{verbatim}

\usepackage{etoolbox}
\makeatletter
\@ifundefined{color@begingroup}%
  {\let\color@begingroup\relax
   \let\color@endgroup\relax}{}%
\def\fix@ieeecolor@hbox#1{%
  \hbox{\color@begingroup#1\color@endgroup}}
\patchcmd\@makecaption{\hbox}{\fix@ieeecolor@hbox}{}{\FAILED}
\patchcmd\@makecaption{\hbox}{\fix@ieeecolor@hbox}{}{\FAILED}

\def\BibTeX{{\rm B\kern-.05em{\sc i\kern-.025em b}\kern-.08em
    T\kern-.1667em\lower.7ex\hbox{E}\kern-.125emX}}
\usepackage{graphicx,epstopdf,multirow,multicol,booktabs,pgfplots}
\usepackage{xcolor}

\usepackage{graphicx,epstopdf,multirow,booktabs}
\usepackage{amsmath,amssymb,amsthm,bm}
\usepackage{cases}
\usepackage[T1]{fontenc}
\usepackage{hyperref}
\allowdisplaybreaks
\hypersetup{
     colorlinks =true,
     linkcolor = blue,
     filecolor = blue,
     citecolor = magenta,      
     urlcolor = magenta,
     }
\usepackage{setspace}
\usepackage{siunitx}
\usepackage{csvsimple}
\usepackage{url}
\usepackage[scaled]{helvet}
\usepackage[utf8]{inputenc}

\usepackage[style=ieee,doi=false]{biblatex}
\AtBeginBibliography{\footnotesize}

\let\labelindent\relax
\usepackage{enumitem, xspace}

\usepackage{accents}
\usepackage{mathtools}

\usepackage{xspace}
\newcommand{\htwo}{H\textsubscript{2}\xspace}

\theoremstyle{remark}
\newtheorem{proposition}{Proposition}
\newtheorem{theorem}{Theorem}

\newtheorem{remark}{Remark}

\def\d{\partial}

\def\mR{\mathbb{R}}

\pgfplotsset{compat=1.14}

\begin{document}
\title{\LARGE \bfseries Modeling and Optimization of Steady Flow of Natural Gas and Hydrogen Mixtures in Pipeline Networks}
\author{Saif R. Kazi $^{\dagger}$, Kaarthik Sundar$^{*}$, Shriram Srinivasan$^{\dagger}$, Anatoly Zlotnik$^{\dagger}$
\thanks{$^{\dagger}$Applied Mathematics and Plasma Physics Group, Los Alamos National
Laboratory, Los Alamos, New Mexico, USA. E-mail: \texttt{\{skazi,shrirams,azlotnik\}@lanl.gov}}\;
\thanks{$^{*}$Information Systems and Modeling Group, Los Alamos National
Laboratory, Los Alamos, New Mexico, USA. E-mail: \texttt{{kaarthik}@lanl.gov}}\;
\thanks{The authors acknowledge the funding provided by LANL's Directed Research and Development (LDRD) projects: ``20220006ER: Fast, Linear Programming-Based Algorithms with Solution Quality Guarantees for Nonlinear Optimal Control Problems'' and ``20230617ER: Efficient Multi-scale Modeling of Clean Hydrogen Blending in Large Natural Gas Pipelines to Reduce Carbon Emissions''. The authors also thank the U.S. Department of Energy's Advanced Grid Modeling (AGM) projects Joint Power System and Natural Gas Pipeline Optimal Expansion and Dynamical Modeling, Estimation, and Optimal Control of Electrical Grid-Natural Gas Transmission Systems. The research work conducted at Los Alamos National Laboratory is done under the auspices of the National Nuclear Security Administration of the U.S. Department of Energy under Contract No. 89233218CNA000001. 
}\;
}
\maketitle

\begin{abstract}
We extend the canonical problems of simulation and optimization of steady-state gas flows in pipeline networks with compressors to the transport of mixtures of highly heterogeneous gases injected throughout a network.  Our study is motivated by proposed projects to blend hydrogen generated using clean energy into existing natural gas pipeline systems as part of efforts to reduce the reliance of energy systems on fossil fuels.  Flow in a pipe is related to endpoint pressures by a basic Weymouth equation model, with an ideal gas equation of state, where the wave speed depends on the hydrogen concentration. At vertices, in addition to  mass balance, we also consider mixing of incoming flows of varying hydrogen concentrations.  The problems of interest are the heterogeneous gas flow simulation (HGFS), which determines system pressures and flows given fixed boundary conditions and compressor settings, as well as the heterogeneous gas flow optimization (HGFO), which extremizes an objective by determining optimal boundary conditions and compressor settings.  We examine conditions for uniqueness of solutions to the HGFS, as well as compare and contrast mixed-integer and continuous nonlinear programming formulations for the HGFO.  We develop computational methods to solve both problems, and examine their performance using four test networks of increasing complexity.
\end{abstract}

\IEEEpeerreviewmaketitle

\section{Introduction} \label{sec:intro}
The ongoing transition away from the use of fossil fuels has compelled the development of technologies for production, transport, and utilization of hydrogen gas for electricity production, transportation fuel, and residential use \cite{haeseldonckx2007use}. However,  natural gas still accounts for a large proportion of primary energy consumption throughout the world, and it is transported from sources to consumers using pipelines.  Interest has developed recently in blending hydrogen generated using clean energy into existing infrastructure designed decades ago for natural gas, so that the significant capital investments made in these systems can continue to generate revenue while supporting the energy transition \cite{witkowski2018analysis}.  Such power-to-gas technologies can also support power transmission systems by using excess energy generated by intermittent and unpredictable wind and solar sources to generate hydrogen \cite{guandalini2017dynamic}.  However, there are multiple technical and conceptual challenges to the development of these technologies \cite{melaina2013blending,hafsi2018hydrogen,schuster2020centrifugal,raju2022}, and the fundamentals of modeling, simulation and optimization of pipelines must be revised to account for location-dependent hydrogen blending.

Natural gas transmission systems are designed to transport processed pipeline quality gas, which is typically 95\% methane, 3-4\% ethane, and 1-2\% other longer chain hydrocarbons.  The actual content of gas transport networks is monitored using gas chromatographs, and is constituted by a mixture of different gas species injected by suppliers that process gas from different formations \cite{folga2007natural}. Isothermal natural gas flow models assume homogeneous gas, and can be formulated in terms of mass flow and density, where pressure is obtained by a nonlinear function of density according to an equation of state and compressibility model \cite{thorley87,gyrya2019explicit}.  In contrast, when natural gas constituents and much lighter hydrogen flow through a network, they meet and mix at network junctions, and this altering chemical composition also alters the physical properties of the gas, which affects aspects of physical transport and energy content in a location-dependent manner \cite{hante2019complementarity}.  The existing models must be extended to account for gas composition, specific calorific value, and gas gravity, which are tracked throughout the network using a mixing model. 

Methodologies for steady-state simulation and optimization of gas pipeline network flows are well-established and validated \cite{koch2015evaluating,bermudez2015simulation}.  These, as well as transient simulation solvers based on well-known techniques \cite{osiadacz1984simulation}, are used widely in the pipeline industry to evaluate capacity.  Efforts to extend analysis of natural gas pipeline flows to account for hydrogen blending were found to be surprisingly challenging in both the steady-state \cite{tabkhi2008mathematical} and transient \cite{zhang2022modelling} regimes.  Though the physical flow modeling is well-understood and validated for single pipelines \cite{chaczykowski2018gas}, a critical difficulty arises in systems with loops when the flow direction is not known \emph{a priori}.  In addition, the physical properties of much lighter hydrogen gas increase the numerical ill-conditioning in equations for flow computations of the blend, compared with pure natural gas.

We extend recent results on natural gas network flow solvers \cite{kekatos2019,kekatos2020,srinivasan2022numerical} and natural gas network flow optimization \cite{koch2015evaluating,wu2017adaptive} to develop robust, scalable simulation and optimization solvers that account for non-uniform injection of hydrogen gas into a pipeline network.  We first consider the heterogeneous gas flow simulation (HGFS), which determines system pressures and flows given fixed boundary conditions and compressor settings, and prove uniqueness of the solution for the special case of tree networks.  We then formulate the heterogeneous gas flow optimization (HGFO), which optimizes the boundary conditions and compressor settings in order to maximize an objective that represents the economic value for users of the pipeline system.  Our main focus here is on modeling and comparing optimization formulation approaches that enable \emph{scalable} computation, and we expect modeling extensions, economic interpretations, and transient problems to be addressed in future studies.

We first present flow modeling for a single pipe in Section \ref{sec:single-pipe}, and then extend this to networks with compressors in Section \ref{sec:network}.  We then state the HGFS problem and examine conditions for uniqueness of its solution in Section \ref{sec:hgfs}, and state the HGFO problem and consider its reformulation using non-smooth equations, complementarity constraints, and mixed integer formulations in Section \ref{sec:hgfo}.  Thereafter, we examine the performance of the developed algorithms using scenarios with an 8-node tree network, an 8-node looped network, the GasLib-11 instance, and the GasLib-40 instance in Section \ref{sec:results}, and finally conclude in Section \ref{sec:conc}.

\section{Steady flow in a pipe} \label{sec:single-pipe}

Modeling and optimizing the flow of gaseous mixtures in pipeline networks is a notoriously challenging problem \cite{van1998gas,chaczykowski2018gas}. Complications arise because of poorly understood interactions between the gases in turbulent flow and the need to account for the relative proportions of gas components at each spatial location, as well as their individual kinematic quantities and constitutive responses \cite{brethouwer1999direct}.  Accordingly, high-fidelity modeling that incorporates most or all of these factors makes network optimization numerically intractable. We narrow our focus  to make a set of assumptions that are justifiable for large scale steady-state network flows and which facilitate models that are tractable for optimization.

The inherent complexities in this setting require more information to be provided in boundary and initial conditions than is required for network flow optimization of a homogeneous gas \cite{rios2015optimization}.  Prior studies have considered transport over networks with a single source of mixed gas \cite{elaoud2017numerical}, so that the mixture is essentially assumed to be homogeneous. Here our goal is to model the flow of a NG-\htwo gas mixture in a network of pipelines where these constituent gases could be injected at some subset of nodes, and where mixtures of \emph{a priori} unknown fractions may be withdrawn at other locations.  Thus, we adopt a simplified view of flow modeling in order to enable generalizable boundary conditions over a network.  

In general, one of three kinds of boundary conditions can be prescribed at a node:  (1) A given quantity of gaseous mixture of known concentration may be injected (``injection nodes'');  (2) A gas mixture of known concentration could be injected at given pressure (``slack nodes''); or (3) A given quantity of the mixture may be withdrawn (``withdrawal nodes''). Note that at withdrawal nodes, we do not specify concentration because we suppose that the mass fractions cannot be controlled locally. Moreover, specifying only pressure but not concentration at a node with net injection to the network is insufficient when there are multiple gas constituents at varying fractions.  Concentration must be specified together with pressure (for a slack node) or flow (for an injection node) if gas flows into the network at that node.  \emph{We suppose that slack nodes are associated with inflows to the network.}  Finally, we remark that although it is possible to have a node where quantities of gaseous mixtures of known concentration may be withdrawn and also injected, we suppose for ease of notation and presentation that individual nodes are subject only to injection or withdrawal.  However,  the results presented in this paper can be extended to that general case as well. 
\subsection{Governing equations for steady flow in a pipe}
\begin{figure}[htb]
    \centering
    \includegraphics[scale=0.8]{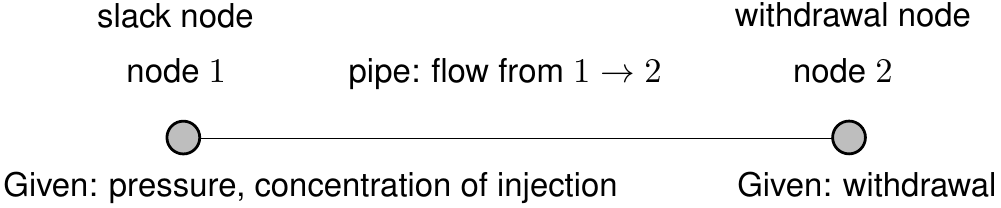}
    \caption{The combination of boundary conditions required for a well-defined problem of flow of NG-\htwo mixture in a single pipe. Any slack node where pressure is specified must be an injection node, and concentration of the injected gas must be specified. For steady flow, gas withdrawn at the other end must necessarily have the same concentration.}
    \label{fig:single_pipe} 
\end{figure}
We first present our assumptions and the resulting governing equations for a single pipe.  For a single pipe as shown in Fig.~\ref{fig:single_pipe}, one end must be a slack node while the other end is a withdrawal node.  This combination of boundary conditions is imposed by the topology.  While the flow solution is degenerate if pressure is unknown at both ends of the pipe, we do not by convention define slack nodes at both ends either.  Next, we expect that if the pipe is long enough, both gas constituents eventually move at the same velocity within the pipe.  Conservation principles ensure that for steady flow under this configuration, the composition at the withdrawal node will be identical to that at the injection.   

Our assumptions on mixing are: (i) the gases are chemically inert, so that there is no chemical reaction, and flow is isothermal; (ii) at nodes, the gases mix quickly to form a homogeneous mixture so that for a long pipe, it may be assumed that constituent gases travel with the mixture velocity; and (iii) we restrict our attention to steady flows.


The isothermal flow of a compressible gas in a pipeline is described by the Euler equations in one dimension \cite{thorley87}.  A number of assumptions are made for large-scale pipeline flow in the typical regime \cite{osiadacz1984simulation,roald2020uncertainty}, which we extend here to the flow of an NG-\htwo blend through a single pipe.  We write mass and momentum balance for each gas in the mixture:
\begin{subequations}
\begin{gather}
  \frac{\d\rho^h}{\d t}+\frac{\d \varphi^h}{\d x} = 0, 
  \label{eq:h2-mass} \\
  \frac{\d\rho^{ng}}{\d t}+\frac{\d \varphi^{ng}}{\d x} = 0, 
  \label{eq:ng-mass} \\
  \frac{\d\varphi^h}{\d t} + \frac{\d p^h}{\d x}  = 
  -\frac{\lambda}{2D}\frac{\varphi^h |\varphi^h|}{\rho^h} + I_f,  		
  \label{eq:h2-momentum} \\
  \frac{\d\varphi^{ng}}{\d t} + \frac{\d p^{ng}}{\d x}  = 
  -\frac{\lambda}{2D}\frac{\varphi^{ng} |\varphi^{ng}|}{\rho^{ng}} - I_f,  		
  \label{eq:ng-momentum}
\end{gather}
\label{eq:euler1}
\end{subequations}
\noindent where $\rho^{h}$, $p^h$, $\varphi^h = \rho^h v$ and $\rho^{ng}$, $p^{ng}$, $\varphi^{ng} = \rho^{ng} v$ are the component densities, pressures and mass fluxes of \htwo and NG in the mixture, respectively, and $v$ is the velocity of the gas.  In equations \eqref{eq:euler1}, these quantities vary in both space and time.   In the momentum balance equations \eqref{eq:h2-momentum} and \eqref{eq:ng-momentum},  $I_f$ represents the interaction force between the two gases. The other parameters in Eq. \eqref{eq:euler1} are the friction factor, $\lambda$ and the diameter, $D$ of the pipe. For simplicity of presentation, it has been assumed that the pipes are horizontal, so that the effects related to change in the elevation of the pipe have been ignored in Eq. \eqref{eq:h2-momentum} and \eqref{eq:ng-momentum}.  We omit the kinematic viscosity terms in the momentum equations, which are negligible in the typical setting of pipeline flows \cite{osiadacz1984simulation}.

In this study, we are interested in the steady-state flow regime. Thus, setting the time derivatives in Eq. \eqref{eq:euler1} to zero and adding the momentum balance equations for \htwo and NG to eliminate $I_f$, we obtain
\begin{subequations}
\begin{gather}
    \varphi^h = \text{ constant, } \varphi^{ng} = \text{ constant, } \label{eq:ss-mass} \\
    \frac{d}{dx}\left(p^h + p^{ng}\right) = -\frac{\lambda}{2D}\left(\frac{\varphi^{h} |\varphi^{h}|}{\rho^{h}} + \frac{\varphi^{ng} |\varphi^{ng}|}{\rho^{ng}}\right). \label{eq:ss-momentum}
\end{gather}
\label{eq:ss-euler}
\end{subequations}
Eq. \eqref{eq:ss-mass} indicates that in a steady flow regime, the mass flux of \htwo and NG stays the same across all cross-sections of the pipe. Now, if $\rho$, $p$, and $\varphi$ denote the density, pressure, and mass flux of the mixture in the pipe, and if $\gamma$ denotes the constant mass fraction of \htwo throughout the pipe (in steady flow, $\gamma$ is constant due to Eq. \eqref{eq:ss-mass}), then we have:
\begin{subequations}
\begin{gather}
\rho^h = \gamma \cdot\rho, ~~ \rho^{ng} = (1-\gamma)\cdot \rho, \label{eq:density-fraction} \\
\varphi^h = \gamma \cdot\varphi, ~~ \varphi^{ng} = (1-\gamma) \cdot\varphi. \label{eq:flux-fraction}
\end{gather}
\label{eq:fractions}
\end{subequations}
The governing equations in Eq. \eqref{eq:ss-euler} are supplemented with the equation of state (EoS)  that relates the mixture density $\rho$ and pressure $p$ of the gas in the form
\begin{equation}
p = p^h + p^{ng} = a_h^2 \rho^h + a_{ng}^2 \rho^{ng} = \rho\left(\gamma a_h^2 + (1 - \gamma) a_{ng}^2 \right). 
\label{eq:eos}
\end{equation}
%
The EoS in Eq. \eqref{eq:eos} assumes that  the \htwo-NG mixture behaves like a mixture of two ideal gases. Here, $a_h$ and $a_{ng}$ denote the speed of sound for the ideal gases \htwo and NG respectively at constant temperature $T$. Eq. \eqref{eq:eos} is obtained by using Eq. \eqref{eq:density-fraction}. Defining a function $V: \mR \rightarrow \mR$ as 
\begin{flalign}
    V(\gamma) \triangleq \gamma a_h^2 + (1 - \gamma) a_{ng}^2, \label{eq:V}
\end{flalign}
we can rewrite Eq. \eqref{eq:eos} as $p = V(\gamma)\cdot \rho$.  Finally, we let $f = A \varphi$ denote the constant mass flow through the cross-sectional area of the pipe $A$. 

We note that the ideal gas assumption is quite strong, because natural gas is highly non-ideal at transmission pipeline pressures, and the state equation for the mixture would require the solution of a system of nonlinear equations in practice.  
For simplicity of exposition, and in order to focus on the role of concentration $\gamma$, we proceed using Eq. \eqref{eq:eos} and relegate the non-ideal extensions to future studies.

We now combine Eq. \eqref{eq:ss-euler}, \eqref{eq:fractions}, and \eqref{eq:eos} to rewrite the governing equation for steady flow of the mixture through a single pipeline in terms of mixture variables $p$, $\rho$, and $f$ as
\begin{gather}
    \frac{dp}{dx} = -\frac{\lambda}{2DA^2} \frac{f|f|}{\rho}, \label{eq:mixture-momentum-balance}
\end{gather}
where the EoS in Eq. \eqref{eq:eos} defines the relationship between $p$ and $\rho$. Integrating Eq. \eqref{eq:mixture-momentum-balance} along the length of the pipe with the end-point data labelled with subscripts $1$ and $2$  and constant mass flow $f$
in the direction from point $x_1$ to $x_2$, the end-point pressures $p_1$ and $p_2$ are related by
\begin{gather}
    p_2^2 - p_1^2 = -\frac{\lambda L}{DA^2} \cdot V(\gamma) \cdot f|f|, \label{eq:steady-physics}
\end{gather}
where $L = |x_2 - x_1|$ is the length of the pipe. In Eq. \eqref{eq:steady-physics}, if $f$ is negative, it implies that the direction of flow is from $x_2$ to $x_1$. If we define a potential function $\pi : \mR \rightarrow \mR$ to have the value $\pi(p) = p^2$, and let $\pi_i$ denote the value  $\pi(p_i)$ then, Eq. \eqref{eq:steady-physics} can be rewritten in terms of the potentials as 
\begin{gather}
    \pi_2 - \pi_1 = -\frac{\lambda L}{DA^2}\cdot V(\gamma) \cdot f|f|. \label{eq:steady-physics-potential}
\end{gather}

\subsection{Non-dimensionalization of the governing equations} \label{subsec:nd}

As is well-known for a single gas \cite{Sundar2018,srinivasan2022numerical}, non-dimensionalization of the governing equations can enable numerical solution techniques to avoid scaling issues.
To that end, we choose nominal length, pressure, density, and velocity $l_0$, $p_0$, $\rho_0$ and $v_0$ respectively, and set the nominal mass flux, area, and mass flow to $\varphi_0 = v_0 \rho_0$, $A_0 = 1$, and $f_0 = A_0 \varphi_0$. Then setting $\bar L = L/l_0$, $\bar p = p/p_0$, $\bar f = f/f_0$, $\bar D = D/l_0$, and $\bar \pi(\cdot) = \pi(\cdot)/p_0^2$,  where the scaling utilizes the factors
$a_0 = \sqrt{a_h \cdot a_{ng}} \approx 672$ \si{\meter\per\second},
$v_0 = a_0 \cdot M \approx 6.72$ \si{\meter\per\second}, $M \approx 0.01$, and
$\bar V = V/(a_0^2)$, Eq. \eqref{eq:steady-physics-potential} reduces to
\begin{gather}
    \bar \pi_2 - \bar \pi_1 = -\frac{\lambda \bar L}{\bar D \bar A^2} \bar V(\gamma) \bar f| \bar f| \left( \frac{v_0^4 \rho_0^2}{M^2 p_0^2}\right), \label{eq:physics-nd}
\end{gather}
where $\bar f$ and $\gamma$ are constants.
Letting 
\begin{flalign}
\bar \beta \triangleq \left( \frac{\lambda \bar L}{\bar D \bar A^2}\right) \cdot \left( \frac{v_0^4 \rho_0^2}{M ^2 p_0^2}\right), \label{eq:multiplier}
\end{flalign}
Eq. \eqref{eq:physics-nd} reduces to 
\begin{flalign}
\bar \pi_2 - \bar \pi_1 = -\bar  \beta \cdot \bar V(\gamma) \cdot \bar f| \bar f|. \label{eq:physics}
\end{flalign}
Here, $\bar \beta$ is analogous to the effective resistance of the pipe. 
While we are free to choose the values of the nominal quantities,  we present guidelines in Sec. \ref{sec:results} for choosing appropriate nominal values that lead to better scaling of the problem, and which thereby ensure better performance of the numerical methods. \emph{In subsequent discussions, for ease of presentation, we shall drop the overbar that designates non-dimensional quantities with the understanding that all quantities are dimensionless.}

\section{Steady flow over a network} \label{sec:network}
A pipeline network consists of nodes, pipes, and compressors. Nodes are physical locations where one or more pipes and/or compressors are joined. At each node, a gas mixture can be either injected or withdrawn. Accordingly, the set of nodes are partitioned into slack, injection, and withdrawal nodes. As detailed in Sec. \ref{sec:single-pipe}, each slack node is a node where the pressure and the concentration of the injected gas is known, an injection node is one where a given quantity of gaseous mixture of known concentration is injected, and a withdrawal node is one where a given quantity of mixture is withdrawn. Both pipes and compressors are associated with a pair of nodes, and each compressor is modeled as a pipe of zero length with a prescribed pressure boost \cite{Hari2021,Borraz2016}. 

Before we present the gas flow problem for a network, we  define the notation.\\ 
\textbf{Sets}
\begin{itemize}
    \item $\mathcal G  = (\mathcal N, \mathcal P \cup \mathcal C)$ - graph of a pipeline network 
    \item $\mathcal N, \mathcal P, \mathcal C$ - sets of nodes, pipes and compressors, respectively
    \item $\mathcal S, \mathcal I, \mathcal W$ - sets of slack, injection and withdrawal nodes, respectively
    \item $(i,j) \in \mathcal P$ - the pipe connecting nodes $i$ and $j$
    \item $(i,j) \in \mathcal C$ - the compressor between nodes $i$ and $j$
\end{itemize}

\noindent \textbf{Variables}
\begin{itemize}
    \item $f_{ij}$ - steady flow through pipe/compressor $(i,j)$
    \item $\gamma_{ij}$ - concentration of \htwo in pipe/compressor $(i,j)$
    \item $\eta_i$ - concentration of \htwo (after mixing) at node $i$
    \item $\pi_i$ - squared pressure at node $i$
    \item $\alpha_{ij}$ - compressor ratio in compressor $(i,j)$
    \item $q^s_i,q^w_i$ - supply and withdrawal flows at injection $i \in \mathcal I$ and withdrawal $i \in \mathcal W$ nodes, respectively
\end{itemize}

\noindent \textbf{Parameters}
\begin{itemize}
    \item $\lambda_{ij}$ - friction factor of pipe $(i,j)$
    \item $A_{ij}$ - cross-sectional area of pipe $(i,j)$
    \item $L_{ij,}D_{ij}$ - length and diameter of pipe $(i,j)$
    \item $\beta_{ij}$ - effective resistance of pipe $(i,j)$
    \item $\eta^s_i$ - concentration of injection at node $i \in \mathcal I \cup \mathcal S$
\end{itemize}
We assume, without loss of generality, that $\mathcal N = \mathcal S \cup \mathcal I \cup \mathcal W$, i.e., the sets of slack, injection and withdrawal nodes partition the node set $\mathcal N$ of the pipeline network. This can always be done by representing nodes where no net gas is being injected or withdrawn as withdrawal nodes with zero net withdrawal. 
The convention for flow direction is as follows: when the gas flows from $i \rightarrow j$ ($j \rightarrow i$) $f_{ij}$ is positive (negative). 
It is assumed for a compressor $(i, j) \in \mathcal C$ that the direction of flow is along the direction of the pressure boost, i.e., $i \rightarrow j$, and hence $f_{ij} \geqslant 0$. 
For each node $i \in \mathcal N$, we assume that all the incoming flows mix homogeneously and result in an effective concentration after mixing, denoted by $\eta_i$ and all the outgoing flows from that node may have different flow rates but the same concentration given by the value of $\eta_i$. The Figs. \ref{fig:injection-mixing} and \ref{fig:withdrawal-mixing} illustrate the equation governing $\eta_i$ values at injection and withdrawal nodes respectively assuming that the flow directions are known.  We now reformulate this idea rigorously. 

The direction of flow determines the concentration $\gamma_{ij}$ through mixing at the vertices, i.e., when $f_{ij} > 0$, we have $\gamma_{ij} = \eta_i$ and when $f_{ij} < 0$, $\gamma_{ij} = \eta_i$. This conditional constraint can be equivalently rewritten by employing a single equation
\begin{flalign}
\gamma_{ij} = H(f_{ij}) \cdot \eta_i + (1-H(f_{ij})) \cdot \eta_j, \label{eq:pipe-concentration}
\end{flalign}
where $H(\cdot)$ denotes the Heaviside step function. Eq. \eqref{eq:pipe-concentration} essentially states that the concentration of \htwo in a pipe is determined by the concentration after mixing at the upstream node where the mixture of gases enters the pipe.

Before we present the formulation of the steady flow problem for a \htwo-NG mixture on a network, we list the assumptions made about  the pipeline network.
\begin{enumerate}[label=(A\arabic*)]
    \item There is exactly one slack node, i.e., $|\mathcal S| = 1$.\label{assumption:slacks}
    \item Compressors cannot be in parallel, i.e., for any two junctions $i$ and $j$ there exists at most one compressor that connects $i$ and $j$. 
    \item Pathological network structures are disallowed, such as compressors in sequence forming a loop.
\end{enumerate}
We remark that assumption in (A1) can be relaxed to $|\mathcal S| \geqslant 1$; nevertheless we do not do so for ease of exposition.
With these assumptions, the network model for NG-\htwo mixed gas transport is presented
\begin{figure}
    \centering
    \includegraphics[scale=0.8]{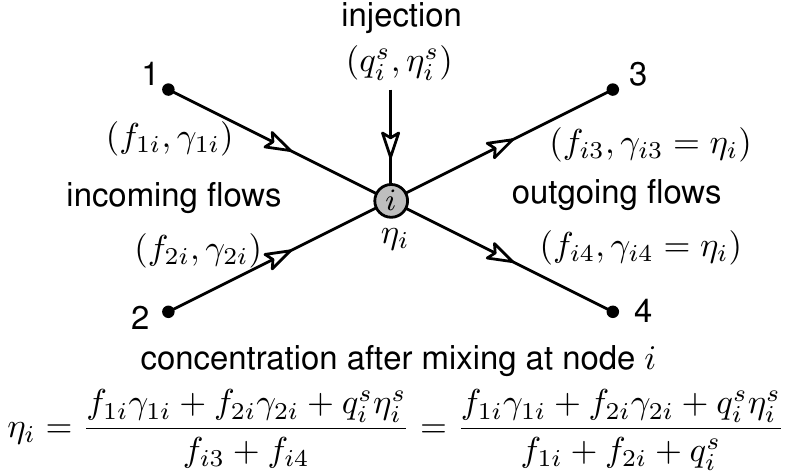}
    \caption{Computation of the mixing concentration $\eta_i$ at a node $i$ for given flow directions when $i$ is either a slack or injection node. If $i$ is a slack node, then $\pi_i$ and $\eta_i^s$ are given, and if $i$ is an injection node, then $q_i^s$ and $\eta_i^s$ are given. The equations are still valid when injection is zero.
    } 
    \label{fig:injection-mixing}
\end{figure}
\begin{figure}
    \centering
    \includegraphics[scale=0.8]{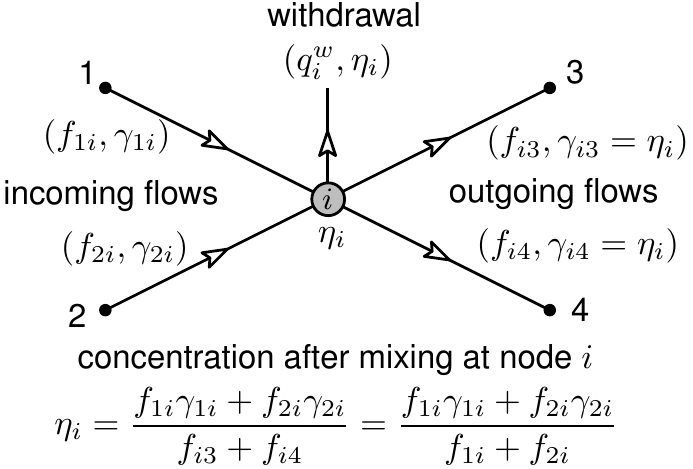} 
    \caption{Computation of the mixing concentration $\eta_i$ for a withdrawal node $i$ for given flow directions. In this case, only $q_i^w$ is given.}
    \label{fig:withdrawal-mixing}
\end{figure}
\vspace{-3mm}

\subsection{Governing equations for the network} \label{subsec:network-gf}

For every pipe $(i, j) \in \mathcal P$, the steady mass flow rate $f_{ij}$, the potential at the ends $i$ and $j$, and the concentration $\gamma_{ij}$ of \htwo  has to satisfy Eq. \eqref{eq:physics}. This equation is rewritten using the network notation above as
\begin{flalign}
\pi_i - \pi_j = \beta_{ij} \cdot V(\gamma_{ij}) \cdot f_{ij} \cdot | f_{ij} |, \label{eq:pipe-physics}
\end{flalign}
where $\beta_{ij}$ is the effective resistance of the pipe.



Each compressor $(i, j) \in \mathcal C$ is associated with pressure boost and concentration equations
\begin{flalign}
\pi_j = \alpha_{ij}^2 \cdot \pi_i \,\,\text{ and }\,\, \gamma_{ij} = \eta_i. \label{eq:compressor-physics}
\end{flalign}
The concentration of \htwo in the compressor $(i, j) \in \mathcal C$ is given by the concentration after mixing at the inlet $i$ because the mass flow $f_{ij}$ through the compressor is along the direction of compression. The nodal balance equations at each node (described below) govern mass flow through the compressor and the concentration after mixing at each node. 

Finally for each node $i \in \mathcal N$, we have two balance equations, one each for (i) the mixture and (ii) the \htwo component (Note that these two equations could have been formulated equivalently  as balance equations for each constituent). To formulate the balance equations, we let $\partial_+i$  and $\partial_-i$ denote the incoming and outgoing subsets of pipes and compressors at node $i$. Then the nodal balance equations for the mixture are
\begin{subequations}
\begin{flalign}
    & \sum_{(j,i) \in \partial_+i} f_{ji} - \sum_{(i,j) \in \partial_-i} f_{ij} = q_i^s ~~ \forall i \in \mathcal I & \label{eq:injection-full-balance} \\
    & \sum_{(j,i) \in \partial_+i} f_{ji} - \sum_{(i,j) \in \partial_-i} f_{ij} = -q_i^w ~~ \forall i \in \mathcal W & \label{eq:withdrawal-full-balance}
\end{flalign}    
\label{eq:full-balance}
\end{subequations}
while the  balance equations for \htwo are given by 
\begin{subequations}
\begin{flalign}
    & \sum_{(j,i) \in \partial_+i} \gamma_{ji} f_{ji} - \sum_{(i,j) \in \partial_-i} \gamma_{ij} f_{ij} = \eta_i^s q_i^s ~~ \forall i \in \mathcal S & \label{eq:slack-h2-balance} \\
    & \sum_{(j,i) \in \partial_+i} \gamma_{ji} f_{ji} - \sum_{(i,j) \in \partial_-i} \gamma_{ij} f_{ij} = \eta_i^s q_i^s ~~ \forall i \in \mathcal I & \label{eq:injection-h2-balance} \\
    & \sum_{(j,i) \in \partial_+i} \gamma_{ji} f_{ji} - \sum_{(i,j) \in \partial_-i} \gamma_{ij} f_{ij} = -\eta_i q_i^w ~~ \forall i \in \mathcal W & \label{eq:withdrawal-h2-balance}
\end{flalign} 
\label{eq:h2-balance}
\end{subequations}
Furthermore, the omission of the nodal balance equation for the mixture at the slack nodes is intentional, for in lieu of these we have: 
\begin{flalign}
\pi_i \text{ and } \eta^s_i \; \mathrm{given} ~~ \forall i \in \mathcal S. \label{eq:slack-pressure}
\end{flalign}

In Eq. \eqref{eq:h2-balance}, the concentration of \htwo withdrawn is dependent on the concentration after mixing at the node $i$.  In fact, Eqs. \eqref{eq:injection-h2-balance}  and \eqref{eq:withdrawal-h2-balance} are exactly what was illustrated in Figs. \ref{fig:injection-mixing} and \ref{fig:withdrawal-mixing} earlier once we recognize that all outgoing edges have the same concentration.


In the next section, we proceed to formulate the heterogeneous gas flow simulation problem on a pipeline network as a boundary  value problem using the equations presented thus far.

\section{Heterogeneous Gas Flow Simulation}  \label{sec:hgfs}
The steady-state heterogeneous gas flow simulation (HGFS) on a pipeline network is a nonlinear system of equations. The variables for the HGFS problem, also refered to as state variables, are (i) the mass flow rate, $f_{ij}$ and the concentration of \htwo, $\gamma_{ij}$ in each pipe and compressor $(i,j)$ in the network,  and (ii) the concentration of \htwo after mixing, $\eta_i$ at every node $i \in N$ and the potential $\pi_i$ at every non-slack node $i \in \mathcal N \setminus \mathcal S$. In total, there are $2 \cdot \left( |\mathcal P| +  |\mathcal C| + |\mathcal N| \right) - |\mathcal S|$ variables. As for the boundary conditions, we are provided with (i) the concentration of injection of \htwo and the total pressure at each slack node in $\mathcal S$, (ii) the net injection and the concentration of injection of \htwo at each injection node in $\mathcal I$, (iii) the net withdrawal at each withdrawal node in $\mathcal W$, and (iv) compression ratio $\alpha_{ij}$ for each compressor. Using these variables and boundary conditions, the steady-state HGFS on a pipeline network takes the form:
\begin{subnumcases} {\label{eq:hgfs}
\text{$\mathtt{SIM}$: } }
\text{Eq. \eqref{eq:pipe-concentration}, \eqref{eq:pipe-physics}} & \text{ -- (pipe physics)}  \\ 
\text{Eq. \eqref{eq:compressor-physics}} & \text{ -- (compressor physics)} \\
\text{Eq. \eqref{eq:full-balance}, \eqref{eq:h2-balance}} & \text{ -- (nodal balance)}  
\end{subnumcases}
with the following boundary conditions 
\begin{subnumcases} {\label{eq:hgfs-given}
\text{$\mathtt{BDY}$: } }
 (q_i^s, \eta_i^s) \quad \forall i \in \mathcal I & \text{ -- (injection node data) } \\ 
    q_i^w \quad \forall i \in \mathcal W  & \text{ -- (withdrawal node data) }\\ 
    \alpha_{ij} \quad \forall (i, j) \in \mathcal C & \text{ -- (compression ratios) } \\
    (\pi_i, \eta_i^s) \quad \forall i \in \mathcal S & \text{ -- (slack node data) }
\end{subnumcases}
The number of nodal balance, pipe physics, and compressor physics equations in $\mathtt{SIM}$ is given by $2 \cdot |\mathcal N| - |\mathcal S|$, $2 \cdot |\mathcal P|$, and $2 \cdot |\mathcal C|$, respectively, resulting in a total of $2 \cdot \left( |\mathcal P| +  |\mathcal C| + |\mathcal N| \right) - |\mathcal S|$ equations. A solution to the HGFS problem on a pipeline network is obtained by solving $\mathtt{SIM}$ with the boundary conditions in $\mathtt{BDY}$. 
Once a solution is obtained, the unknown slack injections can be determined as
\begin{flalign}
    q_i^s = \sum_{(j,i) \in \partial_+i} f_{ji} - \sum_{(i,j) \in \partial_-i} f_{ij} ~~ \forall i \in \mathcal S & \label{eq:slack-full-balance} 
\end{flalign}
Several studies have established uniqueness of solutions for boundary value problems on physical flow networks, including for compressible \cite{vuffray2015monotonicity,misra2020monotonicity,kekatos2019} and incompressible fluid flow \cite{singh2020flow} with actuators. Each of these proofs considers transport of a homogeneous fluid, and relies on the monotonicity of the flow equation on each individual edge.  That is, the difference in potentials on the left hand side of equation \eqref{eq:pipe-physics} can only increase with increasing flow, for fixed uniform concentration $\gamma_{ij}$, and this holds for meshed networks \cite{vuffray2015monotonicity} and transient initial boundary value problems \cite{misra2020monotonicity}.  However, if flow $f_{ij}$ increases and $\gamma_{ij}$ decreases, it may be possible for the potential difference to actually \emph{decrease}.  \emph{Thus establishing uniqueness of boundary value problems for heterogeneous fluid flows on general networks remains open.} 
However, uniqueness of solutions for $\mathtt{SIM}$ can be established in two special settings, (i) for networks with tree structures such as that illustrated in Figure \ref{fig:treenet} below (a proof of this claim is included in the Appendix \ref{sec:uniqueness}) and (ii) for meshed networks with only one slack node and no injection nodes. The proof of the latter follows by observing that when there are no injection nodes in the network, the concentration of \htwo in the entire pipeline network is equal to the slack injection concentration. 
\begin{figure}[h!]
    \centering
    \includegraphics[width=0.6\linewidth]{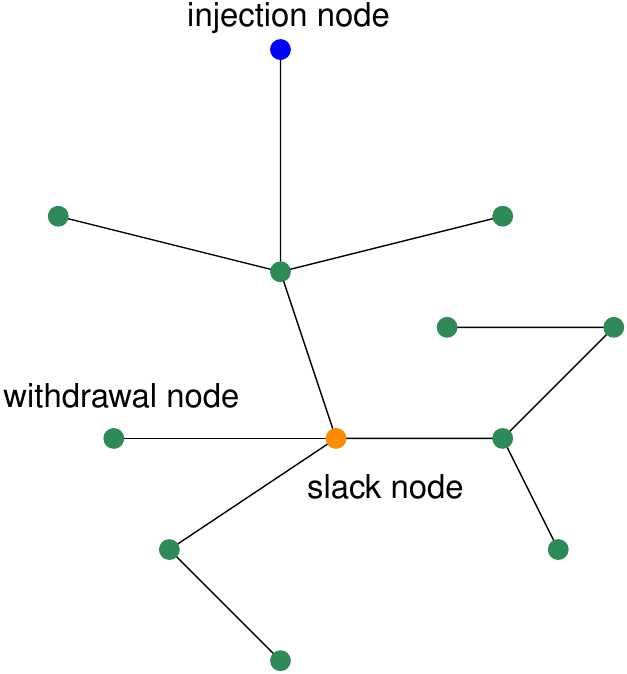}
    \caption{Example of a tree gas pipeline network}
    \label{fig:treenet}
\end{figure}

\section{Heterogeneous Gas Flow Optimization}  \label{sec:hgfo}
We now extend the HGFS problem $\mathtt{SIM}$ with the boundary conditions in $\mathtt{BDY}$  to the heterogeneous gas flow optimization (HGFO) formulation by adding inequality constraints and an objective function. In the context of optimization,  the boundary conditions in $\mathtt{BDY}$  are not known apriori (unlike HGFS)  and need to be determined. Hence, we shall introduce additional decision variables for the same and refer to them as control variables. The goal of the HGFO is to formulate and solve an optimization problem for heterogeneous gas flow over a pipeline network with (i) an appropriate objective function, (ii) the equation system in $\mathtt{SIM}$, and (iii) inequality constraints to model the physical limits of the pipeline network, the injection and withdrawal limits. We now proceed to detail the inequality constraints for the HGFO. 

\subsection{Inequality constraints} \label{sec:inequality}
Operating limits of the pipeline network take the form of simple inequality constraints on the state variables. To that end, constraints on the nodal potential (squared pressures), pipe and compressor mass flow rates, compression ratios, and concentration of \htwo are stated as follows:
\begin{subequations}
    \begin{flalign}
        \pi^{\min} \leqslant \pi_i \leqslant \pi^{\max}, \quad & \forall i \in \mathcal N, \label{eq:pressure-limits} \\ 
        f^{\min} \leqslant f_{ij} \leqslant f^{\max}, \quad & \forall (i, j) \in \mathcal P \cup \mathcal C, \label{eq:flow-limits} \\
        1 \leqslant \alpha_{ij} \leqslant \alpha^{\max}, \quad & \forall (i, j) \in \mathcal C, \label{eq:ratio-limits} \\
        0 \leqslant \eta_i \leqslant \gamma^{\max}, \quad & \forall i \in \mathcal N \setminus \mathcal S, \label{eq:concentration-limits} \\ 
        \eta_i \geqslant 0, \quad & \forall i \in \mathcal S. \label{eq:concentration-limits-slack} 
    \end{flalign}
\label{eq:state-limits}
\end{subequations}
Now proceeding to the control variables,  we start with the withdrawal nodes which represent either residential or industrial users with specified energy demands. For each withdrawal node, we introduce a control decision variable $q_j^w$ that denotes the mass flow rate of mixed gas that is being withdrawn at that node. Since the energy content of the gas being withdrawn at node $j \in \mathcal W$ depends on both the mass flow rate $q_j^w$ and the \htwo concentration $\eta_j$ (the state variable at the node $j$), we define a new variable $g_j$ that represents the energy or heat content of the gas mixture being withdrawn at node $j$ per unit time in terms of $q_j^w$ and $\eta_j$ and its associated inequality constraints as follows:
\begin{subequations}
\begin{flalign}
    & g_j = q_j^w \cdot (R_{\htwo} \cdot \eta_j + R_{NG} \cdot (1-\eta_j)) \quad  \forall j\in\mathcal W, \label{eq:heat_content} \\
    & g_j \leqslant g_j^{\max} \quad \forall j \in \mathcal W. \label{eq:heat_demand}
\end{flalign}
In Eq. \eqref{eq:heat_demand}, $g_j^{\max}$ is interpreted as the requirement for energy at the withdrawal node $j \in \mathcal W$, $R_{\htwo}$ and $R_{NG}$ are calorific values of \htwo and natural gas, respectively.

Similar to the withdrawal nodes, each injection node $j \in \mathcal I$ is associated with a nodal injection control variable denoted by $q_j^s$. Also, the concentration of injection at the injection node is fixed to a pre-specified value. The concentration being fixed to a value stems from the practicality of being able to produce any quantity of mixed gas as a given concentration at the injection node. Nevertheless, we remark that this concentration can also be a variable and the formulation presented in this section will extend to that case as well. The inequality constraint for the injection node is then given by 
\begin{flalign}
    \label{eq:supply_bound}
    0 \leqslant q_j^s \leqslant q_j^{s,\max}, \text{ and } \eta_j^s \text{ is fixed } \forall j \in \mathcal I.
\end{flalign}
Finally, for a slack node, the pressure and the concentration of injection is assumed to be pre-specified, i.e.,
\begin{flalign}
\pi_i, \eta^s_i\text{ are fixed } \quad \forall i \in \mathcal S. \label{eq:slack-node}
\end{flalign}
\end{subequations}

\subsection{Objective} \label{sec:objective}

The objective function represents the value provided by the pipeline to its users; it includes receipts from consumers who purchase energy at withdrawal nodes and subtracts payments for gas flows into injection nodes, together with the operating cost of compressors. 
The inflow concentration is fixed at each injection node, and the offer prices for the supplied mixture  are a linear combination of the fixed prices $C_{\htwo}^s$ \si{\$\per\kilogram} and $C_{NG}^s$ \si{\$\per\kilogram}  for pure \htwo and $NG$ respectively.  The value of the purchase depends on the bids $C_{\htwo}^w$ \si{\$\per\kilogram} and $C_{NG}^w$ \si{\$\per\kilogram} for each gas, which are combined similarly.  Then the revenue created by consumer withdrawals and supplier injections is given by
\begin{align}
\label{eq:obj_gas}
    W^{\mathrm{gas}} & = \sum_{i \in W} \left(C_{\htwo}^w \cdot \eta_i  + C_{NG}^w \cdot (1-\eta_i)\right) \cdot q^w_i  \nonumber \\
    & \quad - \sum_{j \in I} \left( C_{\htwo}^s \cdot \eta_j^s + C_{NG}^s \cdot (1-\eta_j^s)\right) \cdot q_j^s.
\end{align} 
Here we suppose that prices are identical network wide, e.g., assuming a regional hub price,  but the ratio of the offer prices for \htwo and NG may differ from the ratio  of the bid prices to reflect incentives that may alter the  consumption and production costs.  
The compressor work is determined using the adiabatic compression formula
\begin{equation}\label{eq:compressor_work}
    W^{\mathrm{c}} = \frac{286.76 T}{\omega \cdot G}\left(\frac{\alpha^m - 1}{m}\right) \quad \text{\si{\joule\per\kilogram}},
\end{equation}
where $m = \frac{\kappa}{\kappa - 1}$, $G$ and $\kappa$ denote the specific gravity and specific heat capacity ratio of the gas, and $T$ and $\omega$ denote the temperature and compressor efficiency \cite{menon05}. As the physical properties of the gas mixture depend on the concentration, we approximate the specific gravity and specific heat capacity ratio for the mixture by linear interpolation, resulting in
\begin{subequations}\label{eq:mixed_property}
\begin{align}
    G = \gamma \cdot G_{\htwo} + (1-\gamma) \cdot G_{NG}, \\
    \kappa = \gamma \cdot \kappa_{\htwo} + (1-\gamma) \cdot \kappa_{NG},
\end{align}
\end{subequations}
where $G_{\htwo}, G_{NG}, \kappa_{\htwo}$ and $\kappa_{NG}$ are constant parameters. 
This simplification is made to reduce complexity of the resulting optimization problem.  The total objective value produced by operating the gas network is:
\begin{equation} \label{eq:obj_total}
    W_{\mathrm{tot}} =  \delta \cdot W^{\mathrm{gas}} - (1-\delta) \cdot \zeta \cdot \left(\sum_{(i,j) \in C} W_{ij}^{\mathrm{c}} \cdot f_{ij}\right),
\end{equation}
where $\zeta$ is the per-unit electricity cost to run compressors, and $\delta$ is a parameter that can be used to place priority on maximizing deliveries or minimizing operating costs.
The complete optimization problem is written as:
\begin{flalign*}
    \mathtt{OPT} ~~ \triangleq ~~ & \max ~~ W_{\mathrm{tot}} ~~~ \text{ subject to: } \\
    & \text{Eq. in } \mathtt{SIM} \text{ -- (component physics) } \\ 
    & \text{Eq. \eqref{eq:state-limits}} \text{ -- (engineering limits)} \\
    & \text{Eq. \eqref{eq:heat_content}, \eqref{eq:heat_demand}} \text{ -- (withdrawal node constraints)} \\ 
    & \text{Eq. \eqref{eq:supply_bound}} \text{ -- (injection node constraints)} \\ 
    & \text{Eq. \eqref{eq:slack-node}} \text{ -- (slack node constraints)} \\ 
    & \text{Eq. \eqref{eq:obj_gas}, \eqref{eq:compressor_work}, \eqref{eq:mixed_property}} \text{ -- (revenue and costs)}
\end{flalign*}
The solution to problem $\mathtt{OPT}$ yields the state variables in $\mathtt{SIM}$ and  determines the boundary data in $\mathtt{BDY}$. 

\begin{remark}
In the simulation problem, the balance equations at the slack nodes \eqref{eq:slack-full-balance} are not used since the slack injection itself is unknown, and come into play only as a post-processing step to determine the unknown slack injections. However, in the optimization setting, one could append \eqref{eq:slack-full-balance} to the existing constraints in $\mathtt{OPT}$ without adverse results. In fact, if relaxations are used to solve $\mathtt{OPT}$, then appending these (linear) constraints to further constrain the search space is advantageous.
\end{remark}

\subsection{Reformulations of $\mathtt{OPT}$}\label{Reformulation}
The $\mathtt{OPT}$ is a Nonlinear Program (NLP) and the non-differentiable Heaviside function appears in  the pipe concentration equation in Eq. \eqref{eq:pipe-concentration}. For the purposes of this article, we are concerned with solving this NLP to local optimality. In general, local solvers for NLPs cannot deal with non-differentiable functions and hence, in this section, we examine several possible reformulations of Eq. \eqref{eq:pipe-concentration} to enable direct application of local optimization solvers. \\

\noindent \textbf{Integer reformulation}:  
The obvious reformulation of the Heaviside function is by the introduction of binary variables. To that end we let $y_{ij}$ be a binary variable which takes a value $1$ when $H(f_{ij}) = 1$ and $0$, otherwise. Using this additional binary variable, Eq. \eqref{eq:pipe-concentration} can be reformulated by the equations
\begin{subequations}\label{eq:integer}
\begin{gather}
    \gamma_{ij} = y_{ij} \cdot \eta_i + (1-y_{ij}) \cdot \eta_j, \\
    -M \cdot (1-y_{ij}) \leqslant f_{ij} \leqslant M \cdot y_{ij}, \\
    y_{ij} \in \{0,1\},
\end{gather}
\end{subequations}
where $M$ is a bound on the mass flow, $f_{ij}$ in the pipe. In this context, we remark that using this reformulation will make $\mathtt{OPT}$ a mixed-integer nonlinear program (MINLP) for which there are even fewer off-the-shelf local optimization solvers than for NLPs.\\ 

\noindent \textbf{Non-smooth reformulation}: Here, we propose to reformulate Eq. \eqref{eq:pipe-concentration} using two non-smooth equations
\begin{subequations}\label{eq:nonsmooth}
\begin{align}
    (f^2_{ij} + f_{ij} |f_{ij}|)\cdot(\gamma_{ij}-\eta_i) = 0, \label{eq:i-nonsmooth}\\
    (f^2_{ij} - f_{ij} |f_{ij}|)\cdot(\gamma_{ij}-\eta_j) = 0. \label{eq:j-nonsmooth}
\end{align}
\end{subequations}
We see that when $f_{ij} > 0$, Eq. \eqref{eq:i-nonsmooth} will yield $\gamma_{ij} = \eta_i$ and alternatively, when $f_{ij} < 0$, Eq. \eqref{eq:j-nonsmooth} will yield $\gamma_{ij} = \eta_j$ which is exactly the condition that is modeled by Eq. \eqref{eq:pipe-concentration}. The reformulation in Eq. \eqref{eq:nonsmooth} would result in an NLP and this can readily be handled by off-the-shelf local optimization solvers. \\

\noindent \textbf{Complementarity constraints}: Before we present this reformulation, we recall the definition of complementarity constraints. A complementarity constraint enforces that two variables are complementary to each other; i.e., that the following conditions hold for scalar variables $x$ and $y$:
$$ x \cdot y = 0, \quad x \geqslant 0, \quad y \geqslant 0. $$ 
The condition above is sometimes expressed more compactly as $0 \leqslant x \perp y \geqslant 0$. Although, complementarity constraints might seem to be a difficult class of nonlinear constraints to tackle algorithmically, we present such a reformulation for Eq. \eqref{eq:pipe-concentration} because some off-the-shelf NLP solvers are equipped to handle them in an efficient manner. The reformulation is
\begin{subequations}\label{eq:complementarity}
\begin{align}
    & f_{ij} = s^1_{ij} - s^2_{ij}, \\
    & \gamma_{ij} = \nu_{ij} \cdot \eta_i + (1-\nu_{ij}) \cdot \eta_j, \\
    & 0 \leqslant s^1_{ij} \perp (1 - \nu_{ij}) \geqslant 0, \\
    & 0 \leqslant s^2_{ij} \perp \nu_{ij} \geqslant 0, 
\end{align}
\end{subequations}
where, the sets of variables $s^1_{ij}$, $s^2_{ij}$, and $\nu_{ij}$ are auxiliary variables introduced for the reformulation. 

Although all these reformulations still yield non-smooth problems, they can be solved by smooth approximations or relaxation methods \cite{scholtes,Kai_and_Saif}. The reformulations are equivalent but will have different numerical properties and performance depending on the type of optimization solver used, as seen in the next section on computational results. 

\section{Results} \label{sec:results}
In this section, we present extensive computational results to corroborate the utlity of both the simulation and the optimization problem formulations in modeling the flow of NG-\htwo gas mixture in a network of pipelines. To that end, the proposed mixed gas network simulation and optimization are tested on four example networks (i) an 8-node network with a loop, (ii) an 8-node tree network, (iii) the GasLib-11 network, and (iv) the GasLib-40 network. The GasLib-11 and the GasLib-40 are standard benchmark pipeline networks \cite{gaslib} used for testing models for transport of natural gas through a network of pipelines. We now present a brief description for each network. 


\subsection{Description of the test networks} \label{subsec:network-data}
The 8-node network with a loop is a single gas network with one slack node (J1) and two withdrawal nodes (J3 and J5); the network is shown in Fig. \ref{fig:8-node}. The slack node pressure for this network is set at 5 \si{\mega\pascal} and the concentration of \htwo injected at this node is set to 10\%, i.e., $\eta^s = 0.1$. The 8-node tree network is similar to the 8-node network with loop, except for the absence of the pipe P\textsubscript{3}, as shown in Fig. \ref{fig:8-node-tree}. The modified GasLib-11 network (see Fig. \ref{fig:GasLib-11}) is a tree with one slack and one injection node (J6 and J7, respectively). The network has three withdrawal nodes (J9, J10 and J11, respectively). The \htwo concentration of the injected flows at the slack and the injection nodes are fixed to 5\% and 7.2\%, respectively.
The GasLib-40 network (see Fig. \ref{fig:GasLib-40}) is another example from \cite{gaslib} with one slack and two injection nodes (J38, J39 and J40, respectively) and twenty nine withdrawal nodes. The slack and injection flow concentrations are fixed ($\eta_{38}^s = 0.05$ at J38,  $\eta_{39}^s = 0.098$ at J39, \&  $\eta_{40}^s = 0.055$ at J40). Next, we present the table of parameters and constants that are used both in the HGFS and HGFO problem formulations.

\begin{figure*}[htbp]
    \centering
    \includegraphics[width=.8\linewidth]{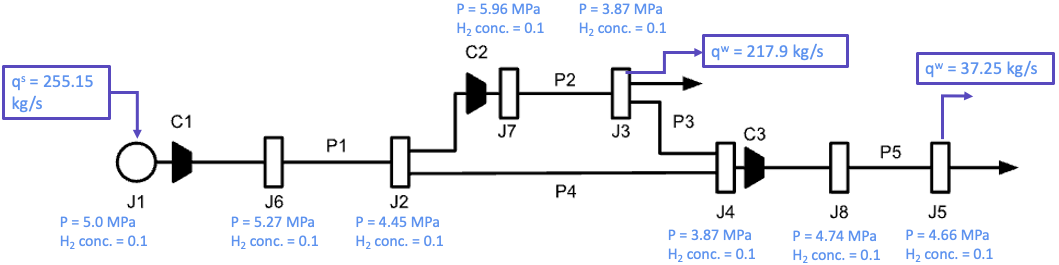} \vspace{-1ex}
    \caption{8-node network with a loop.  The nodal pressure, concentration, and injection/withdrawal for the optimal solution (Table \ref{tab:8-node}) are shown.} 
    \label{fig:8-node}
\end{figure*}
\begin{figure*}[htbp]
    \centering
    \includegraphics[width=.8\linewidth]{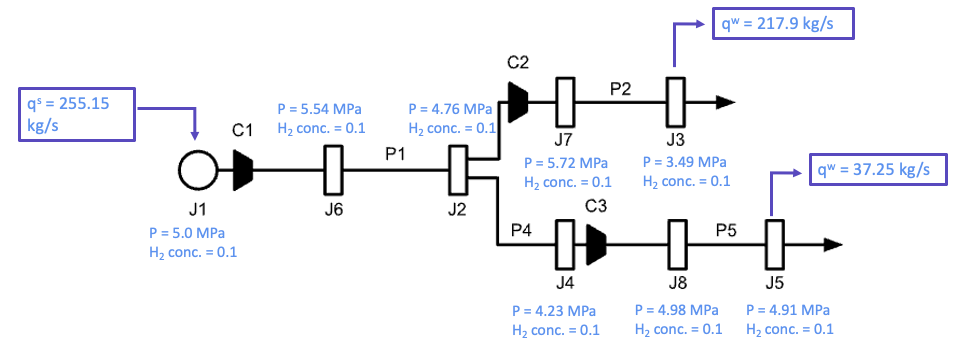} \vspace{-1ex}
    \caption{8-node tree network.  The nodal pressure, concentration, and injection/withdrawal for the optimal solution (Table \ref{tab:8-node-tree}) are shown.}
    \label{fig:8-node-tree}
\end{figure*}
\begin{figure*}[htbp]
    \centering
    \includegraphics[width=.9\linewidth]{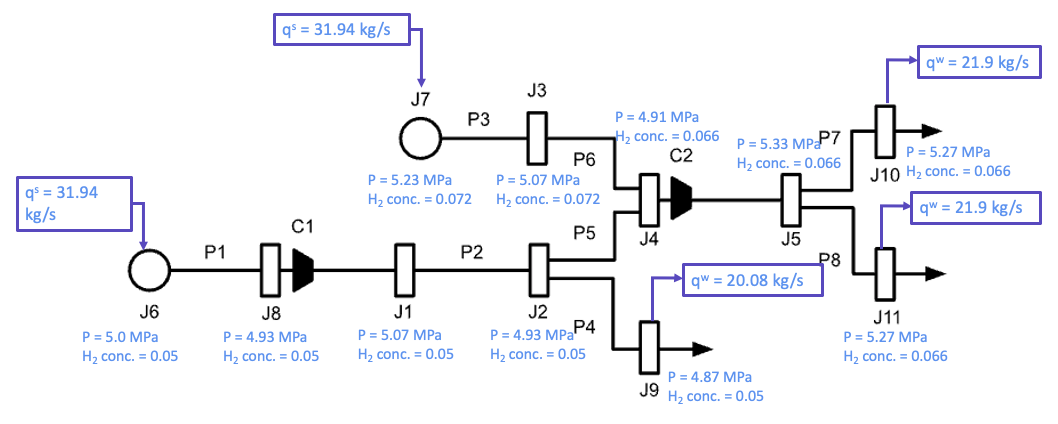}
    \caption{GasLib-11 network. The nodal pressure, concentration, and injection/withdrawal for the optimal solution (Table \ref{tab:GasLib-11}) are shown.}
    \label{fig:GasLib-11}
\end{figure*}
\begin{figure*}[htbp]
    \centering
    \includegraphics[width=\linewidth]{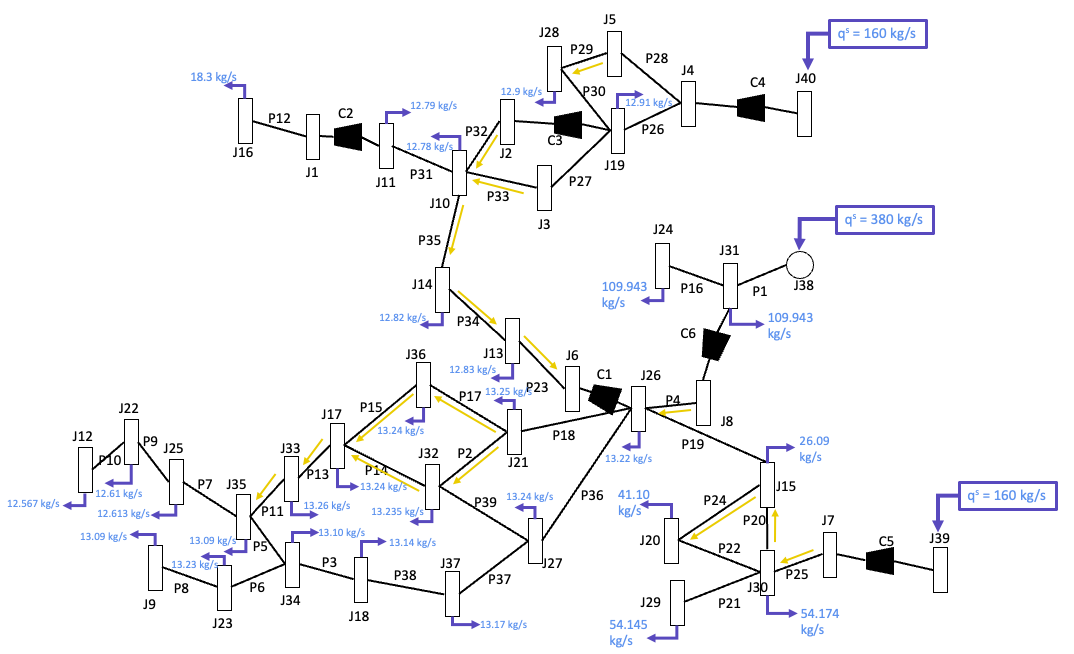}
    \caption{GasLib-40 Network. The nodal pressure, concentration, and injection/withdrawal for the optimal solution (Table \ref{tab:GasLib-40}) are shown.}
    \label{fig:GasLib-40}
\end{figure*}

\subsection{Parameter values and cost coefficients} \label{subsec:params}
The physical parameters and properties for NG and \htwo, along with other cost coefficients for optimization are given in Table \ref{tab:properties}.  These values are used for all test cases considered. 

\begin{table}[htbp]
    \centering
    \begin{tabular}{cc}
    \toprule
    Parameters & Values\\
    \midrule
    Temperature (T) & 288.75 \si{\kelvin} \\
    Speed of Sound (a) & $\sqrt{RT/M_w}$ \\
    $a_{\htwo}, \enskip a_{NG}$ & 1092\si{\meter\per\second},  372 \si{\meter\per\second} \\
    $R_{\htwo}, \enskip R_{NG}$ & 141.8 \si{\mega\joule\per\kilogram}, 44.2 \si{\mega\joule\per\kilogram}\\
    $G_{\htwo}, \enskip G_{NG}$ & 0.0696, 0.6 \\
    $\kappa_{\htwo}, \enskip \kappa_{NG}$ & 1.4, 1.33\\
    $C^s_{\htwo}, \enskip C^s_{NG}$ & 8 \si{\$\per\kilogram}, 2 \si{\$\per\kilogram} \\
    $C^w_{\htwo}, \enskip C^w_{NG}$ & 15 \si{\$\per\kilogram}, 5 \si{\$\per\kilogram} \\
    $\zeta$ & 0.13 \si{\$\per\J} \\
    $\delta$ & 0.95 \\ 
    $\gamma^{\max}$ & 0.10 \\
    \bottomrule
    \end{tabular}
    \caption{Parameter values and cost coefficients}
    \label{tab:properties}
\end{table}

Finally, the maximum heat content ($g_i^{max}$) for each withdrawal node is also pre-determined by calculating the heat content of the maximum withdrawal flow ($q^{max}$) at maximum \htwo concentration ($\gamma^{max}$),
\begin{equation*}
    g^{max}_i = \left(R_{\htwo} \cdot \gamma^{max} + R_{NG} \cdot (1-\gamma^{max})\right) \cdot q_i^{max}.
\end{equation*}

\subsection{Implementation details} \label{subsec:implementation}
We now present a brief overview of the implementation details for both the HGFS and HGFO problems. Both the problems are implemented using the Julia \cite{julia} programming language. The simulation uses a Newton Trust-Region method that is made available through the Julia package entitled ``NLSolve.jl''. As for the optimization problem, it is formulated using the algebraic modeling language JuMP \cite{jump} and solved using off-the-shelf open source and commercial solvers (IPOPT \cite{ipopt} and KNITRO \cite{knitro}, respectively). While we use both the solvers to compare and contrast the computational performance, the integer reformulation can only be solved by KNITRO since it implements a branch-and-bound algorithm to handle integer decision variables. On the other hand, IPOPT is a primal-dual interior-point solver and hence, is capable of solving only the non-smooth and complementarity reformulations of HGFO. In order to understand the relative impact of different solvers and algorithms on the optimization problem formulation and its reformulations,  we compare the performance of different solution methods for the three reformulations  in Sec. \ref{Reformulation}, based on the objective value, number of iterations, and the computational solution times. In this context, it is important to note that both KNITRO and IPOPT are local solvers, i.e., they can only find a locally optimal solution to the HGFO reformulations. Hence, different algorithmic choices in the solvers could result in different objective values; nevertheless, it is guaranteed that each of them is a locally optimal solution to the problem. This fact further motivates the need to study the impact of the different solvers and algorithms on the reformulations. In our presentation of the results, we start with HGFO since the solutions from these problems are used to generate the boundary data in $\mathtt{BDY}$ for the HGFS. 

\subsection{Results for the HGFO} \label{subsec:hgfo-results}
In this section, we present the objective value, the number of iterations and the solution times obtained by using different combination of solvers and algorithms on all the reformulations. The combination of solvers and algorithms used are as follows: 
\begin{enumerate}[label=(\roman*)]
\item KNITRO with a direct interior-point method \cite{knitro_ip-direct}, 
\item KNITRO with a an active-set method \cite{knitro-activeset}, 
\item KNITRO with a Sequential Quadratic Programming (SQP) method, and finally 
\item IPOPT with a direct interior-point method. 
\end{enumerate}
We note that the integer reformulation of the HGFO can only be solved using (i) in combination with a branch-and-bound algorithm which is also implemented in KNITRO \cite{knitro-minlp}, whereas the rest of the reformulations can be solved by all the four solver-algorithm combinations. The Tables \ref{tab:8-node} -- \ref{tab:GasLib-40} present the performance of each solver-algorithm combination on all the four test networks. 

\begin{table}[htbp]
    \centering
    \footnotesize
    \begin{tabular}{cccc}
    \toprule 
    \multirow{2}{*}{Solver-algo.} & \multicolumn{3}{c}{Reformulation results - (Obj./Iter./Time(s))} \\ 
    \cmidrule(lr){2-4} 
    & Non-smooth & Complementarity & Integer \\
    \midrule 
    (i) & \textbf{31.5}/53/0.008 & \textbf{31.5}/16/0.34 & 0.0/-/1.24 \\
    (ii) & \textbf{31.5}/6/0.003 & \textbf{31.5}/6/1.19 & -na- \\
    (iii) & \textbf{31.5}/30/0.184 & 0.0/59/0.48 & -na- \\
    (iv) & \textbf{31.5}/520/0.113 & iter. limit & -na- \\
    \bottomrule
    \end{tabular}
    \caption{HGFO results for the 8-node network with a loop.}
    \label{tab:8-node}
\end{table}

\begin{table}[htbp]
    \centering
    \footnotesize
    \begin{tabular}{cccc}
    \toprule 
    \multirow{2}{*}{Solver-algo.} & \multicolumn{3}{c}{Reformulation results - (Obj./Iter./Time(s))} \\ 
    \cmidrule(lr){2-4} 
    & Non-smooth & Complementarity & Integer \\
    \midrule 
    (i) & \textbf{31.5}/3459/0.562 & \textbf{31.5}/32/0.005 & 26.89/-/0.53 \\
    (ii) & \textbf{31.5}/5/0.002 & \textbf{31.5}/6/0.004 & -na- \\
    (iii) & \textbf{31.5}/66/0.32 & 0.0/7/1.67 & -na- \\
    (iv) & \textbf{31.5}/235/0.38 & iter. limit & -na- \\
    \bottomrule
    \end{tabular}
    \caption{HGFO results for the 8-node tree network.}
    \label{tab:8-node-tree}
\end{table}

\begin{table}[htbp]
    \centering
    \footnotesize
    \begin{tabular}{cccc}
    \toprule 
    \multirow{2}{*}{Solver-algo.} & \multicolumn{3}{c}{Reformulation results - (Obj./Iter./Time(s))} \\ 
    \cmidrule(lr){2-4} 
    & Non-smooth & Complementarity & Integer \\
    \midrule 
    (i) & \textbf{7.52}/52/0.23 & 3.81/113/0.019 & \textbf{7.52}/-/0.009 \\
    (ii) & \textbf{7.52}/29/0.011 & \textbf{7.52}/8/1.19 & -na- \\
    (iii) & \textbf{7.52}/9/1.44 & 3.71/50/1.66 & -na- \\
    (iv) & \textbf{7.52}/48/0.011 & \textbf{7.52}/197/0.048 & -na- \\
    \bottomrule
    \end{tabular}
    \caption{HGFO results for the GasLib-11 network.}
    \label{tab:GasLib-11}
\end{table}

\begin{table}[htbp]
    \centering
    \footnotesize
    \begin{tabular}{cccc}
    \toprule 
    \multirow{2}{*}{Solver-algo.} & \multicolumn{3}{c}{Reformulation results - (Obj./Iter./Time(s))} \\ 
    \cmidrule(lr){2-4} 
    & Non-smooth & Complementarity & Integer \\
    \midrule 
    (i) & \textbf{82.52/}49/0.06 & \textbf{82.52}/299/0.3 & \textbf{82.52}/-/1.325 \\
    (ii) & \textbf{82.52}/15/0.03 & \textbf{82.52}/16/0.04 & -na- \\
    (iii) & 16.76/30/3.67 & 0.54/85/19.7 & -na- \\
    (iv) & \textbf{82.52}/290/0.16 & \textbf{82.52}/2391/2.415 & -na- \\
    \bottomrule
    \end{tabular}
    \caption{HGFO results for the GasLib-40 network.}
    \label{tab:GasLib-40}
\end{table}

In all the tables, the best objective value obtained for each network is highlighted in bold font-face.  It is clear  that the computational performance of the non-smooth reformulation is better than that of  the integer and the complementarity constraint reformulations. Among the solver-algorithm combinations, it is observed that the KNITRO-interior-point combination (i) gives a good balance of computation time and number of iterations. The optimal operating flows and pressures throughout the network for each test case are shown in the Figs. \ref{fig:8-node}--\ref{fig:GasLib-40}. Note that for the GasLib-11 network in Fig. \ref{fig:GasLib-11}, pipeline flows with different concentrations (P5 and P6) mix at the node J4 and form a mixed concentration flow downstream.  As for the GasLib-40 network in Fig. \ref{fig:GasLib-40}, 16 pipes out of 29 in total have negative flows (i.e. in opposite direction with respect to graph orientation) denoted by green arrows. This demonstrates the need to include the directionality constraint \eqref{eq:pipe-concentration} and its reformulations. We remark that, if we solve the HGFO problem with fixed direction of flows as opposed to the manner in which it was done for the results in table \ref{tab:GasLib-40} with directionality constraints, the objective value decreases from 82.52 to 25.55 (almost 70\% decrease in value).

\subsection{Results for the HGFS} \label{subsec:results-hgfs}

We examine solution of the heterogeneous gas flow simulation problem (HGFS) \eqref{eq:hgfs} where the boundary data in $\mathtt{BDY}$ are provided by the solution of HGFO.  The problem is solved using the non-smooth  formulation \eqref{eq:nonsmooth} with the trust-region method.

\begin{table}[htbp]
    \centering
    \footnotesize
    \begin{tabular}{ccc}
    \toprule 
    Network & \# Iterations & Solve time (s) \\
    \midrule 
    8-node & 73 & 0.066 \\ 
    8-node tree & 5 & 0.065 \\ 
    GasLib-11 & 3 & 0.062 \\ 
    GasLib-40 & 30 & 0.18 \\ 
    \bottomrule
    \end{tabular}
    \caption{HGFS results for all the four test networks.}
    \label{simulationresults}
\end{table}

The state variables (node pressure and concentration, pipe flows and concentration) are verified to be equal to the solution obtained by the optimization problem. The numeric results listed in Table \ref{simulationresults} show that the computational time required to solve  the system is negligible in all instances.

\section{Conclusions and Future Work} \label{sec:conc}
In this paper, we derive the formulation for steady flow of hydrogen and natural gas mixtures in a pipeline network including composition tracking variables. We introduce the heterogeneous gas flow simulation (HGFS) problem, and show uniqueness of the solution for tree networks.  We examine several ways of formulating the direction dependent concentration tracking constraint in the heterogeneous gas flow optimization (HGFO) problem using continuous and integer formulations. A comparison of these formulations using standard off-the-shelf commercial and open-source solvers using a suite of different algorithms shows the effectiveness of the non-smooth reformulation over the integer and complementarity formulations. The example of the GasLib-40 network shows the importance of modeling the bi-directionality constraint and accounting for flow direction changes in meshed networks.  Future studies are expected to extend the heterogeneous gas flow formulations for simulation and optimization to transient flows and optimal control. In addition, analysis of optimality conditions and dual variables could be used to develop methods for locational pricing of heterogeneous gas transport,  specifically to price the benefits of reduction in carbon emissions with biomethane and hydrogen utilization.

\printbibliography

\section{Appendix}
\subsection{Proof of uniqueness for a tree} \label{sec:uniqueness}
We start by a the following well-understood result for a tree network with balanced flows. 
\begin{proposition} \label{prop:tree-1}
If a tree network has non-zero (balanced) flow passing through it, there must be  at least one vertex with injection and all incident edges on the vertex with flow is directed out of the vertex and at least one vertex with withdrawal and all incident edges on the vertex with flow is directed into of the vertex.
\end{proposition}
\begin{proof}
The  proof follows by induction starting with the simplest network containing two vertices and recognizing that as the network grows, introduction of new edges should not form cycles.
\end{proof}
\begin{theorem} \label{thm:uniqueness}
For a tree network with $\mathcal N$ nodes, $\mathcal N-1$ edges, and \emph{a single slack node}, the solution to the HGFS problem is unique if it exists. 
\end{theorem}
\begin{proof}
The flow balance equations \eqref{eq:full-balance} can be written in matrix form as $\bm A \bm f = \bm q$, where the matrix $\bm A$ is the reduced incidence matrix for the graph obtained by discarding the rows corresponding to the slack vertices.
Moreover, the structure of the incidence matrix directly implies the flows must be balanced, i.e., the sum of all injections and withdrawals must be zero \cite{srinivasan2022numerical}. 
For a tree with a single slack node,  $\bm A$ is invertible and hence a unique solution for $\bm f$ exists. Thus, the upstream and downstream directions are now known for all edges.
If we can now show that the solution to \eqref{eq:h2-balance} (the concentration) is also unique, then it will conclude the proof.

To that end, suppose that there are two solutions $\bm \gamma^1, \bm \gamma^2$, i.e., an edge exists with valid concentration $\gamma^1_{ij}$ as well as $\gamma^2_{ij}$. The value is determined by mixing at the upstream vertex, so there must be two valid nodal concentrations at the vertex.  
(A) If the vertex is such that it has only outgoing edges, then it must either  have given non-zero injection with specified concentration or   be a slack node having unknown injection at specified concentration. In either case, the non-unique nodal concentration contradicts the given boundary condition.
(B) However, if the vertex has incoming edge(s), then there must be non-unique concentration along at least one incoming edge. We may choose one such edge without loss of generality and follow the upstream direction to arrive at another vertex with non-unique concentration and again consider possibilities (A) and (B).
Since the graph is a tree, the proposition \ref{prop:tree-1} assures us that by repeated consideration of (B), possibility (A) will eventually occur, thus concluding the proof.

Now armed with the knowledge that both the flows and concentrations are unique, it is easy to see that the pressures must also be unique if they exist.  One can see this by starting at the slack node and computing pressures of each node using \eqref{eq:pipe-physics} while traversing a path.
\end{proof}

\end{document}